\RequirePackage{fix-cm} 
\pdfoutput=1

\documentclass[smallextended]{svjour3}       

\makeatletter
\let\cl@chapter\undefined
\makeatletter

\smartqed  

\usepackage{mathptmx} 
\usepackage{graphicx} 
\usepackage{multirow}
\usepackage{hyperref}
\usepackage{color}
\usepackage[table]{xcolor}
\definecolor{gray-table-row}{gray}{0.90}
\usepackage{algorithm}
\usepackage{algpseudocode}
\usepackage{verbatim}
\usepackage{mathtools}
\usepackage{amsmath}
\usepackage{amssymb}
\usepackage[nameinlink]{cleveref}

\newcommand{\isep}{\mathrel{{.}\,{.}}\nobreak} 
\newcommand{\orcid}[1]{\href{https://orcid.org/#1}{[#1]}}

\journalname{Mathematical Programming Computation}

\begin{document}

\title{Enhanced Formulation for Guillotine 2D Cutting Problems}


\author{Henrique Becker\orcid{0000-0003-3879-2691}\\ \and Olinto Araujo\orcid{0000-0003-1136-5032}\\ \and Luciana S. Buriol\orcid{0000-0002-9598-5732}}

\authorrunning{H. Becker et al.} 

\institute{
	Henrique Becker \at
	Federal University of Rio Grande do Sul (UFRGS), Av. Bento Gonçalves, 9500, Porto Alegre, RS, Brazil\\
	\email{hbecker@inf.ufrgs.br}           
\and
	Olinto Araujo \at
	Federal University of Santa Maria (UFSM), Av. Roraima, 1000, Santa Maria, RS, Brazil\\
	\email{olinto@ctism.ufsm.br}           
\and
	Luciana S. Buriol \at
	Federal University of Rio Grande do Sul (UFRGS), Av. Bento Gonçalves, 9500, Porto Alegre, RS, Brazil\\
	\email{buriol@inf.ufrgs.br}           
}

\date{Received: date / Accepted: date}

\maketitle

\begin{abstract}
We advance the state of the art in Mixed-Integer Linear Programming (MILP) formulations for Guillotine 2D Cutting Problems by (i) adapting a previously-known reduction to our preprocessing phase and by (ii) enhancing a previous formulation by cutting down its size and symmetries.
Our focus is the Guillotine 2D Knapsack Problem with orthogonal and unrestricted cuts, constrained demand, unlimited stages, and no rotation -- however, the formulation may be adapted to many related problems.
The code is available.
Concerning the set of 59 instances used to benchmark the original formulation, and summing the statistics for all models generated, the enhanced formulation has only a small fraction of the variables and constraints of the original model (respectively, 3.07\% and 8.35\%).
The enhanced formulation also takes about 4 hours to solve all instances while the original formulation takes 12 hours to solve 53 of them (the other six runs hit a three-hour time limit each).
We integrate, to both formulations, a pricing framework proposed for the original formulation; the enhanced formulation keeps a significant advantage in this situation.
Finally, in a recently proposed set of 80 harder instances, the enhanced formulation (with and without the pricing framework) found: 22 optimal solutions for the unrestricted problem (5 already known, 17 new); 22 optimal solutions for the restricted problem (all are new and they are not the same 22 of the optimal unrestricted solutions); better lower
bounds for 25 instances; better upper bounds for 58 instances.

\end{abstract}

\section{Introduction}

The problem we focus on this work is the Guillotine 2D Knapsack Problem with orthogonal (and unrestricted) cuts, constrained demand, unlimited stages, and no rotation.
We will refer to this specific variant as G2KP.
If we further qualify the G2KP, we only mean to discard the qualifiers above that directly conflict with the extra qualifiers, if any.
The G2KP is an NP-hard problem~\cite{russo:2020}.
The work also focuses on obtaining optimal solutions for this problem through Mixed-Integer Linear Programming (MILP).
We propose two simple but effective enhancements regarding a state-of-the-art MILP formulation for the G2KP (which may also benefit some closely related problem variants).

\subsection{Explanation of the problem and some close variants}

An instance of the G2KP consists of: a rectangle of length~\(L\) and width~\(W\) (hereafter called \emph{original plate}); a set of rectangles~\(\bar{J}\) (also referred to as \emph{pieces}) where each rectangle~\(j \in \bar{J}\) has a length~\(l_j\), a width~\(w_j\), a profit~\(p_j\), and a demand~\(u_j\)
We assume, without loss of generality, that all such values are positive integers.

The G2KP seeks to maximise the profit of the pieces obtained by cutting the original plate.
The \emph{guillotine} qualifier means every cut always go from one side of a plate to other; a cut never stops or starts from the middle of a plate.
A consequence of this rule is that we often do not obtain the pieces directly from the original plate.
We cut the original plate into intermediary plates \(j \in J\), \(J \supseteq \bar{J}\), which we further cut following the same rule.

If we do not cut a plate further, then it is either: thrown away as trim/waste for no profit; or, if it has the same size as a piece, sold by the piece profit value.
\emph{Orthogonal cuts} are always parallel to one side of a plate (and perpendicular to the other).
Consequently, any intermediary plate~\(j\) is always a rectangle, and have a well-defined~\(l_j\) and~\(w_j\).
\emph{Unrestricted cuts} mean we are allowed to make horizontal (vertical) cuts different from the width (length) of a piece.
We will mention the G2KP with restricted cuts further in the text, as solving it exactly is a costly but high-quality primal heuristic for the G2KP.

\emph{Constrained demand} means we can sell at most~\(u_j\) copies of piece~\(j\).
The G2KP with \emph{unconstrained demand} is not NP-hard; exact algorithms of pseudo-polynomial time complexity exist~\cite{beasley:1985}.
Consequently, interesting G2KP instances have~\(u_j < \lceil L / l_j \rceil \times \lceil W / w_j \rceil \) for at least one piece~\(j\) (if not for all pieces).
\emph{Unlimited stages} means there is no limit to the number of times the guillotine switches between horizontal and vertical orientations.
In the exact \(k\)-staged G2KP, the guillotine is switched at most \(k-1\) times.
Consequently, a 2-staged G2KP has all cuts in some orientation before any cuts in the other orientation.
The non-exact \(k\)-staged G2KP adds one extra stage in which the only cuts allowed are the ones that trim plates to the size of pieces.
The \emph{no-rotation} qualifier means we never switch length and width during the cutting process; especially, we cannot sell a plate~\(j\) as a piece of length~\(w_j\) and width~\(l_j\).

The literature further distinguishes between \emph{weigthed} and \emph{unweighted} problem variants.
In the weighted variant, pieces have an arbitrary profit value, while in the unweighted variant the profit value is always equivalent to the piece area.
Consequently, the unweighted variant is equivalent to minimising waste and is a particular case of the weighted variant.
Any algorithm that solves the weighted variant (as is our case) can solve the unweighted variant by setting the piece profit values to their areas.

While our work focuses on this specific problem, the enhanced formulation we present may be readily adapted to, at least, the Guillotine 2D version of the following problems: the Cutting Stock Problem (and the Bin Packing Problem); the Strip Packing Problem; the Multiple Knapsack Problem; the Orthogonal Packing Problem; and the variant allowing rotation for all previously mentioned problems.
See~\cite{furini:2016} for more details.
We do not define or further discuss these problems or variants in this work.

\subsection{Motivation}

The G2KP and its closely related variants are of undisputable interest of the industry, especially wood, paper, metal, and glass cutting industries.
The vast and growing literature on the subject examined by~\cite{iori:2020} and by~\cite{russo:2020} is enough proof of such interest.
To pick a single recent case study see~\cite{clautiaux:2019}, which solves a unique variant of the Guillotine 2D Cutting Stock Problem for a glass factory manufacturing double-paned windows.

We focus on MILP as the solving method (instead of \emph{ad hoc} solutions) because its adaptability amplifies the value of any enhancements we obtain.
A better MILP formulation means:
a better solving procedure for the many (already mentioned) closely related problem variants;
a better continuous relaxation for computing an optimistic guess on the objective value of all these variants (some \emph{ad hoc} algorithms of the literature use MILP solvers to compute their bounds);
not only a better exact method but also a better base for heuristics or anytime procedures;
an immediate benefit from parallelisation, automatic problem decomposition, and solver-implemented heuristics;
and, finally, better ageing of the method over the years through the current trends of multiple-cores processors and ever-advancing solver performance.

\subsection{Contributions and paper outline}

The main contributions of this work are:
an enhanced MILP formulation based on a previous state-of-the-art formulation, its proof of correctness, and empirical evidence of its better performance;
a straigthforward adaptation of a previously known reduction procedure for both the original and the enhanced formulations, and empirical evidence of its positive impact on their performance;
finally, we present new upper and lower bounds, as well as optimal values, for many recently proposed hard instances from~\cite{velasco:2019}.
For such, we reimplement a state-of-the-art MILP formulation and an optional pricing procedure used by it.
This reimplementation allows us to compare both approaches fully.
All code used is available in the first author's repository ({\small\url{https://github.com/henriquebecker91/GuillotineModels.jl/tree/0.2.4}}).

We organise the rest of the paper the following way:
\autoref{sec:related_work} analyses how our work interacts with the pre-existing literature;
\autoref{sec:psn} introduces some mathematical concepts and explains the reduction we adapted from the literature;
\autoref{sec:enhanced_model} describes our enhanced formulation and briefly explains how it differs from the state-of-the-art formulation it is based on;
\autoref{sec:experimental_results} presents our experiments and the empirical results we derive from them;
\autoref{sec:conclusions} delivers our conclusions and suggests future work.

\section{Related work}
\label{sec:related_work}

We do not intend to provide a full overview of the literature, instead we:
refer to surveys; discuss only closely related works and how they interact with our contributions; and opportunely point out missing connections between related works.

Two relevant surveys have come out recently.
\cite{iori:2020} catalogues exact methods and relaxations for 2D cutting problems including guillotine problems.
\cite{russo:2020} reviews the literature of our particular problem at length -- there G2KP is referred to as Constrained 2D Cutting or C2DC.
Moreover, \cite{russo:2020}~points out three strategies employed by previous exact solving methods which cause loss of optimality.
Our work does not employ any of these three strategies.
One of these strategies is a dominance rule that is valid for the unconstrained case but not for the constrained case.
In 1972, \cite{herz:1972}~proposed a dominance rule for the G2KP with unconstrained demand based on the same principle and warned about the possibility of misusing the rule in the constrained case.

The first MILP formulation dealing with guillotine cuts and unlimited stages was proposed by~\cite{messaoud:2008} in 2008.
The problem considered by~\cite{messaoud:2008} is the Strip Packing Problem, but adapting the formulation to the knapsack variant would not change its fundamentals.
Previously, \cite{lodi:2003}~had proposed two MILP formulations for 2-staged G2KP.
As noted by~\cite{belov_thesis:2003}, modeling \(k\)-staged cuts for \(k \geq 3\) (unlimited stages included) was considered difficult at the time.
The size of most \(k\)-staged formulations is exponential on the number of stages (i.e.,~\(k\)).
The formulation of~\cite{messaoud:2008} had about \(3n^4/4\) variables and \(2n^4\) constraints (where \(n\) is the number of pieces) it also employed, according to the authors, a ``very loose linear relaxation'' due to which ``the practical interest of this formulation is still limited''.
The characterization of guillotine cuts proposed by~\cite{messaoud:2008} seems to have been simultaneously proposed by~\cite{pisinger:2007}. 

The first MILP formulation specifically for the G2KP was proposed by~\cite{furini:2016} in 2016.
An extended version of~\cite{furini:2016} appears in~\cite{dimitri_thesis} (a PhD thesis).
Their formulation has pseudo-polynomial size, \(O((L + W) \times L \times W)\)~variables and \(O(L \times W)\) constraints, and its relaxation provides a stronger bound than~\cite{messaoud:2008}.
It was the first formulation able to solve medium-sized instances of the literature.
Besides the formulation, \cite{furini:2016}~proposes two reductions and one pricing procedure; all of these are reimplemented by our work.
They also present and prove a theorem to assure the correctness of one of their reductions~(\emph{Cut-Position}).
A similar theorem and proof appear in~\cite{song:2010}.

In this work, we propose an enhanced formulation based on the one from \cite{furini:2016} mentioned above.
A significant advantage of our enhancement is to avoid the enumeration of any cuts after the middle of a plate.
This advantage appears in many works since~\cite{herz:1972}.
Recently, \cite{delorme:2019} adapted a formulation for the one-dimensional Cutting Stock Problem to obtain this same advantage.
However, the way \cite{delorme:2019}~changes their formulation to obtain this advantage is not the same as our approach.

The most recent MILP formulations for the G2KP come from three works by Martin et alii~\cite{martin:2020:models,martin:2020:bottom,martin:2020:top}.
These formulations are compared against the formulation of~\cite{furini:2016}.
We base our enhanced formulation on~\cite{furini:2016} and also compare against it.
The formulations of ~\cite{martin:2020:models,martin:2020:bottom,martin:2020:top}
have a looser relaxation bound compared to~\cite{furini:2016}, but perform better than~\cite{furini:2016} in instances for which~\cite{furini:2016} has a much larger number of variables.
Considering the instances used in~\cite{furini:2016}, our enhanced formulation dominates the formulation of~\cite{furini:2016}.
Our formulation also dramatically improves the running times of instances in which the formulation of~\cite{furini:2016} performed worse than \cite{martin:2020:models,martin:2020:bottom,martin:2020:top} (e.g., the gcut1--gcut12 instances).
Consequently, while it may be interesting for completeness sake, we do not compare against the formulations proposed in~\cite{martin:2020:models,martin:2020:bottom,martin:2020:top}.


\section{Notation, Discretization, and Plate-Size Normalization}
\label{sec:psn}

The performance of solving methods for cutting and packing problems often heavily depends on the number of (cut/packing) positions considered.
Since the seminal works of~\cite{cw:1977} and~\cite{herz:1972}, solving methods avoid considering each possible position, but instead consider only a subset necessary to guarantee optimality.
The literature includes many such subsets, which are often referred to as \emph{discretizations}.
The most common way of computing these discretizations are Dynamic Programming (DP) algorithms.
These DP algorithms usually only take a small fraction of the running time, but the size of the position subset outputted by them strongly affects the time spent by the rest of the solving method.

Both~\cite{furini:2016} and our enhanced formulation have one constraint for each attainable distinctly-sized plate and one variable for each potential cut over each of these plates.
Therefore, eliminating a single cutting position has the following effects:
\textbf{(i)} it removes one variable for each distinctly-sized plate that allowed that cutting position;
\textbf{(ii)} if that cutting position was the only way to produce some distinctly-sized plates\footnote{Note that the same cutting position, when applied to distinctly-sized plates, may generate different children.}, then it also removes the constraints associated with these plates;
\textbf{(iii)} if (ii) excludes one or more constraints/plates, then it also excludes all variables representing possible cuts over the excluded plates;
\textbf{(iv)} finally, if (iii) eliminates one or more variables/cuts, then it may trigger (ii) again (i.e., other plates stop being attainable), cyclically.

In this work, the only cut subset (discretization) considered are the canonical dissections of~\cite{herz:1972}, hereafter referred to as \emph{normal cuts} instead.
We acknowledge the existence of stricter discretizations: the raster points of~\cite{terno:1987,guntram:1966}, the regular normal patterns of~\cite{boschetti:2002} (named this way by~\cite{cote:2018}), and the Meet-in-the-Middle (MiM) of~\cite{cote:2018}.
The reasons for our choice of discretization are numerous:
it works well with the \emph{Plate-Size Normalization} procedure we describe below;
it is the same discretization employed by~\cite{furini:2016} (from which we base our enhanced formulation on);
MiM main gain is reducing the number of cut positions after the middle of a plate, which our enhanced formulation already discards anyway;
the regular normal patterns compute a distinct subset-sum for each pair of plate and piece, which we consider excessive (there may exist hundreds of thousands of intermediary plate possibilities);
finally, the raster points complicate our proofs and our \emph{Plate-Size Normalization} weakens its benefits.

The set~\(O = \{v, h\}\) denotes the cut orientation: \(v\) is vertical (parallel to width, perpendicular to length); \(h\) is horizontal (parallel to length, perpedicular to width).
Let us recall that the demand of a piece~\(i \in \bar{J}\) is denoted by~\(u_i\).
If we define the set of pieces fitting a plate~\(j\) as~\(I_j = \{i \in \bar{J} : l_i \leq l_j \land w_i \leq w_j \}\), we can define~\(N_{jo}\) (i.e., the set of the normal cuts of orientation~\(o\) over plate~\(j\)) as:

\begin{equation}
N_{jo}= \left\{
\begin{array}{lllr}
  \{q: 0 < q < l_j; & \forall i \in I_j, \exists n_i \in [0 \isep u_i], q = \sum_{i\in I_j} n_i l_i \} & \quad \text{if } o = v,\\
  \{q: 0 < q < w_j; & \forall i \in I_j, \exists n_i \in [0 \isep u_i], q = \sum_{i\in I_j} n_i w_i \} & \quad \text{if } o = h.
\end{array}\right.
\end{equation}

The sets defined above never include cuts at the plate extremities (i.e., \(0\), \(l_j\) for \(N_{jv}\), and \(w_j\) for \(N_{jh}\)).
Any of these cuts will always create (i)~a~zero-area plate and (ii)~a~copy of the plate that is being cut.
Consequently, these cuts only add symmetries and may be disregarded.

The goal of the \emph{Plate-Size Normalization} procedure we propose is to reduce the number of distinctly-sized plates considered.
Fewer distinctly-sized plates mean fewer constraints and trigger the same cascading effect described by items (ii)--(iv) above.
The property exploited by the procedure is already known and similarly exploited by~\cite{alvarez:2009} and by~\cite{dolatabadi:2012}.
We state the property as:

\begin{proposition}
\label{pro:normalization}
Given a plate~\(j\), \(l_j\) may always be replaced by \(l^\prime_j = max\{q : q \in N_{kv}, q \leq l_j\}\) in which \(w_k = w_j\) but \(l_k > l_j\), without loss of optimality.
The analogue is valid for the width.
\end{proposition}

We do not replicate any proof here. We can then define: 

\begin{definition}
The length of a plate~\(j\) is considered normalized if, and only if, \(l_j = l^\prime_j\).
The analogue is valid for the width.
The size of a plate is normalized if, and only if, both its length and its width are normalized.
\end{definition}

The \emph{Plate-Size Normalization} procedure we propose consists only of replacing every non-size-normalized plate enumerated by their normalized counterpart.
The number of distinctly-sized plates diminishes because the procedure replaces many plates of distinct but similar dimensions by a single plate.
The only extra effort added by \emph{Plate-Size Normalization} consists of binary searches over~\(N_{ko}\) sets for each plate~\(j\).
A suitable \(N_{ko}\) set for each plate~\(j\) was already computed by the plate enumeration procedure before introducing the \emph{Plate-Size Normalization} (no extra effort required).

\begin{remark}
If a normal cut~\(q\) divides the size-normalized plate~\(j\), the first child is always size normalized, but the second child may not be size normalized.
\end{remark}

\begin{example}
Given two pieces with \(l = [5, 7]\), \(u = [2, 3]\), and a size-normalized plate of length~\(21\), a normal cut at~\(12\) creates a non-normalized second child of length~\(9\). 
\end{example}

\section{Our changes to Furini's model}
\label{sec:enhanced_model}


The formulation proposed in~\cite{furini:2016} is elegant: the pieces are just intermediary plates that may be sold.
Our contribution consists of changes to both the preprocessing step and to the formulation.
These changes significantly reduce the number of variables.
Differently, these changes deepen the distinction between plates and pieces and, consequently, may be regarded as sacrificing some elegance for performance.
The essentials of the formulation remain the same and, for this reason, we consider the model presented here as an enhanced model, not an entirely new model.


The cut enumeration in~\cite{furini:2016} excludes some symmetric cuts; that is, if two different cuts create the same set of two child plates, then the symmetric cut in the second half of the plate may be ignored.
Differently,~\cite{nicos:1977} disregards \emph{all} cuts after the middle of the plate because of symmetry.
If~\cite{furini:2016} would do the same as~\cite{nicos:1977} it could become impossible to trim a plate to the size of a piece.
For example, if there was a piece with length larger than half the length of a plate, and such plate has no normal cut with the exact length of the needed trim, then the piece could not be extracted from the plate, even if the piece fits the plate.
The goal of our changes is to reduce the number of cuts (i.e., model variables) by getting closer to the symmetry-breaking rule used in~\cite{nicos:1977} without loss of optimality.


\subsection{The enhanced formulation}
\label{sec:enhanced}

Our changes to the formulation are restricted to replacing the set of integer variables~\(y_j, i \in \bar{J},\) with a new set of variables~\(e_{ij}, (i, j) \in E, E \subseteq \bar{J} \times J\), and the necessary adaptations to accomodate this change.
In the original formulation, \(y_i\) denoted the number of times a plate~\(i\) was sold as the piece~\(i\), in this case, the plate always had the exact size of the piece.
Our \emph{extraction variables}~\(e_{ij}\) denote a piece~\(i\) was extracted from plate~\(j\), which size may differ from the size of the piece.
For convenience, we also define \(E_{i*} = \{ j : \exists~(i, j) \in E \}\) and \(E_{*j} = \{i : \exists~(i, j) \in E \}\).
The set \(O = \{h, v\}\) denotes the horizontal and vertical cut orientations.
The set \(Q_{jo}\) (\(\forall j \in J, o \in O\)) denotes the set of possible cuts (or cut positions) of orientation~\(o\) over plate~\(j\).

The parameter~\(a\) is a byproduct of the plate enumeration process.
If cutting a plate~\(k \in J\) with a cut of orientation~\(o \in O\) at position~\(q \in Q_{jo}\) adds a plate~\(j \in J\) to the stock, then~\(a^o_{qkj} = 1\); otherwise~\(a^o_{qkj} = 0\).
The description of this parameter in~\cite{furini:2016} has a typo, as pointed out by~\cite{martin:2020}:
``[...] there is a typo in their definition of parameter~\(a^o_{qkj}\), as the indices~\(j\) and~\(k\) seem to be exchanged.''.

In a valid solution, the value of \(x^o_{qj}\) is the number of times a plate~\(j \in J\) is cut with orientation~\(o \in O\) at position~\(q \in Q_{jo}\); while the value of~\(e_{ij}\) is the number of sold pieces of type~\(i \in \bar{J}\) that were extracted from plates of type~\(j \in J\).
The plate~\(0 \in J\) is the original plate, and it may also be in~\(\bar{J}\), as there may exist a piece of the same size as the original plate.

\begin{align}
\mbox{max.} &\sum_{(i, j) \in E} p_i e_{ij} \label{eq:objfun}\\
\mbox{s.t.} &\sum_{o \in O}\sum_{q \in Q_{jo}} x^o_{qj} + \sum_{i \in E_{*j}} e_{ij} \leq \sum_{k \in J}\sum_{o \in O}\sum_{q \in Q_{ko}} a^o_{qkj} x^o_{qk} \hspace*{0.05\textwidth} & \forall j \in J, j \neq 0,\label{eq:plates_conservation}\\
	& \sum_{o \in O}\sum_{q \in Q_{0o}} x^o_{q0} + \sum_{i \in E_{*0}} e_{i0} \leq 1 &,\label{eq:just_one_original_plate}\\
	& \sum_{j \in E_{i*}} e_{ij} \leq u_i & \forall i \in \bar{J},\label{eq:demand_limit}\\
	& x^o_{qj} \in \mathbb{N}^0 & \forall j \in J, o \in O, q \in Q_{jo},\label{eq:trivial_x}\\
	& e_{ij} \in \mathbb{N}^0 & \forall (i, j) \in E.\label{eq:trivial_e}
\end{align}

The objective function maximizes the profit of the extracted pieces~\eqref{eq:objfun}.
Constraint~\eqref{eq:plates_conservation} guarantees that for every plate~\(j\) that was further cut or had a piece extracted from it (left-hand side), there must be a cut making available a copy of such plate (right-hand side).
One copy of the original plate is available from the start~\eqref{eq:just_one_original_plate}.
The amount of extracted copies of some piece type must respect the demand for that piece type (a piece extracted is a piece sold)~\eqref{eq:demand_limit}.
Finally, the domain of all variables is the non-negative integers~\eqref{eq:trivial_x}-\eqref{eq:trivial_e}.

\subsection{The revised variable enumeration}
\label{sec:var_enum}

The variable enumeration described in~\cite{furini:2016} employs some rules to reduce the number of variables; they are symmetry-breaking, \emph{Cut-Position}, and \emph{Redundant-Cut}.
The two last rules are not discussed here; \cite{furini:2016}~proves their correctness and they do not conflict with the enhanced model.

The use of the \(x\)~variables does not change from the original formulation to our revised formulation -- however, the size of the enumerated set of variables changes.
Our revised enumeration does not create any variable~\(x^o_{jq}\) in which \((o = h \land q > \lceil w_j / 2 \rceil) \lor (o = v \land q > \lceil l_j / 2 \rceil)\).

The original formulation has variables~\(y_i\), \(i \in \bar{J}\), while the revised formulation replaces them with variables~\(e_{ij}\), \((i, j) \in E\), \(E \subseteq \bar{J} \times J\).
Set~\(\bar{J} \times J\) is orders of magnitude larger than~\(\bar{J}\).
Consequently, set~\(E\) must be a small subset to avoid having a revised model with more variables than the original.
A suitable subset may be obtained by a simple rule: \((i, j) \in E\) if, and only if, packing piece~\(i\) in plate~\(j\) does not allow any other piece to be packed in~\(j\).


\subsection{The proof of correctness}

The previous section presented a detailed explanation of the changes to the formulation and variable enumeration.
This section proves such changes do not affect the correctness of the model.
In~\cite{furini:2016}, only the perfect symmetries described below are removed.
Our changes may be summarized to:

\begin{enumerate}
\item There is no variable for any cut that occurs after the middle of a plate.
\item A piece may be obtained from a plate if, and only if, the piece fits the plate, and the plate cannot fit an extra piece (of any type).
\end{enumerate}

The second change alone cannot affect the model correctness.
The original formulation was even more restrictive in this aspect:
a piece could only be sold if a plate of the same dimensions existed.
In our revised formulation there will always exist an extraction variable in such case:
if a piece and plate match perfectly, there is no space for any other piece, fulfilling our only criteria for the existence of extraction variables.
Consequently, what needs to be proved is that:

\begin{theorem}
\label{the:enhanced_correctness}
Without changing the pieces obtained from a packing, we may replace any normal cut after the middle of a plate by a combination of piece extractions and cuts at the middle of a plate or before it.
\end{theorem}


\begin{proof}
This is a proof by exhaustion. The set of all normal cuts after the middle of a plate may be split into the following cases:
\begin{enumerate}
  \item The cut has a perfect symmetry. \label{case:perfectly_symmetric}
  \item The cut does not have a perfect symmetry.
  \begin{enumerate}
    \item Its second child can fit at least one piece. \label{case:usable_second_child}
    \item Its second child cannot fit a single piece.
    \begin{enumerate}
      \item Its first child packs no pieces. \label{case:no_pieces}
      \item Its first child packs a single piece. \label{case:one_piece} 
      \item Its first child packs two or more pieces. \label{case:many_pieces}
    \end{enumerate}
  \end{enumerate}
\end{enumerate}

We believe to be self-evident that the union of~\cref{case:perfectly_symmetric,case:usable_second_child,case:no_pieces,case:one_piece,case:many_pieces} is equal to the set of all normal cuts after the middle of a plate. We present an individual proof for each of these cases.

\begin{description}
\item[\Cref{case:perfectly_symmetric} -- \textbf{The cut has a perfect symmetry.}]
If two distinct cuts have the same children (with the only difference being the first child of one cut is the second child of the other cut, and vice-versa), then the cuts are perfectly symmetric.
Whether a plate is the first or second child of a cut does not make any difference for the formulation or for the problem.
If the cut is in the second half of the plate, then its symmetry is in the first half of the plate.
Consequently, both cuts are interchangeable, and we may keep only the cut in the first half of the plate.
\item[\Cref{case:usable_second_child} -- \textbf{Its second child can fit at least one piece.}]
\autoref{pro:normalization} allows us to replace the second child by a size-normalized plate that can pack any demand-abiding set of pieces the original second child could pack.
The second child of a cut that happens after the middle of the plate is smaller than half a plate, and its size-normalized counterpart may only be the same size or smaller.
So the size-normalized plate could be cut as the first child by a normal cut in the first half of the plate.
Moreover, the old first child (now second child) have stayed the same size or grown (because the size-normalization of its sibling), which guarantee this is possible.

\item[\Cref{case:no_pieces} -- \textbf{Its first child packs no piece.}]
If both children of a single cut do not pack any pieces, then the cut may be safely ignored.
\item[\Cref{case:one_piece} -- \textbf{Its first child packs a single piece.}]
First, let us ignore this cut for a moment and consider the plate being cut by it (i.e., the parent plate).
The parent plate either: can fit an extra piece together with the piece the first child would pack, or cannot fit any extra pieces.
If it cannot fit any extra pieces, this fulfills our criteria for having an extraction variable, and the piece may be obtained through it.
The cut in question can then be disregarded (i.e., replaced by the use of such extraction variable).
However, if it is possible to fit another piece, then there is a normal cut in the first half of the plate that would separate the two pieces, and such cut may be used to shorten the plate.
This kind of normal cuts may successively shorten the plate until it is impossible to pack another piece, and the single piece that was originally packed in the first child may then be obtained employing an extraction variable.
\item[\Cref{case:many_pieces} -- \textbf{Its first child packs two or more pieces.}]
If the first child packs two or more pieces, but the second child cannot fit a single piece (i.e., it is waste), then the cut separating the first and second child may be omitted and any cuts separating pieces inside the first child may still be done.
If some of the plates obtained by such cuts need the trimming that was provided by the omitted cut, then these plates will be packing a single piece each, and they are already considered in~\cref{case:one_piece}.
\end{description}

Given the cases cover every cut after the middle of a plate, and each case has a proof, then follows that \autoref{the:enhanced_correctness} is correct. \qed

\end{proof}

\section{Experimental results}
\label{sec:experimental_results}

There are three formulation implementations that provide data used in our comparisons:
\emph{original} refers to the implementation presented in~\cite{furini:2016,dimitri_thesis};
\emph{faithful} refers to our reimplementation of \emph{original};
\emph{enhanced} refers to our enhanced formulation presented in~\autoref{sec:enhanced_model}.
The \emph{original} implementation was not available\footnote{
	We asked the authors of~\cite{furini:2016} for the \emph{original} implementation and Dimitri Thomopulos informed us it was not available.
}.
Consequently, all data relative to \emph{original} presented in this work comes from~\cite{dimitri_thesis}.
Both \emph{faithful} and \emph{enhanced} data were obtained by runs using the setup described in~\autoref{sec:setup}.

Each formulation may be modified by applying any combination of the following optional procedures:
\emph{priced} -- refer to the pricing procedure described in~\cite{furini:2016,dimitri_thesis};
\emph{normalized} -- the plate-size normalization procedure described in~\autoref{sec:psn};
\emph{warmed} -- the MIP models solved were warm-started with a solution found by a previous step;
\emph{Cut-Position} and \emph{Redundant-Cut} -- are reduction procedures described in~\cite{furini:2016,dimitri_thesis}, that may be enabled and disabled individually.
For each experiment described in the next sections, if we do not mention a procedure, then it is disabled.
The term \emph{restricted priced} refers to the model for the restricted version of the problem that is solved inside the pricing procedure mentioned above.
Consequently, for each run of a \emph{priced} variant, there will be a \emph{restricted priced} run with the same combination of optional procedures.
The differences between the \emph{restricted priced} and the (unrestricted) \emph{priced} models are mainly that:
(i) the \emph{restricted priced} model never has a horizontal (vertical) cut that does not match the width (length) of a piece;
(ii) the \emph{restricted priced} model is MIP-started with the solution of an heuristic (described in~\cite{furini:2016}) while the \emph{priced} model is MIP-started with the solution of the \emph{restricted priced} model;
(iii) the distinct solutions used to MIP-start the respective models are also used as the lower bound for the pricing procedure (details in~\cite{furini:2016}).

The goal of the pricing procedure is to remove unneeded variables from the model.
However, the priced model often ends up with unneeded constraints and variables due to pricing.
This effect is similar to the one described by items (ii)--(iv) in~\autoref{sec:psn}: if some variables (i.e., cuts) are removed, then some plates are never produced (i.e., some constraints just fix their variables to zero), consequently other variables/cuts become impossible, recursiverly.
The effort to remove such unnecessary variables and constraints is negligible.
The algorithm used is similar to finding the connected subgraph in the directed hypergraph defined by the variables/cuts (edges) and constraints/plates (nodes) starting from the original plate.
In \emph{priced} variants of \emph{faithful} and \emph{enhanced} this \emph{purge} procedure is done unless stated otherwise.
Our experiments will show that this \emph{purge} drastically reduces the number of variables and constraints, but has almost no effect on the running times.
Nonetheless, we encourage future comparisons to implement this \emph{purge} procedure, as it helps determine the real size of the solved models.

Each experiment fills a gap for the next experiments:
\autoref{sec:lp_method} explains the choice of LP algorithms made in all remaining experiments;
\autoref{sec:faithful_reimplementation} provides evidence that \emph{faithful} is on par with \emph{original}, allowing us to use it as a replacement;
\autoref{sec:comparison} compares \emph{faithful} to \emph{enhanced} and shows the value of our contributions (namely, the \emph{normalize} procedure and the \emph{enhanced} formulation);
\autoref{sec:new_results} applies the methods with best results in the last experiment to prove new optimal values and bounds for harder instances.

\subsection{Setup}
\label{sec:setup}

Every experiment in this work uses the following setup unless stated otherwise.
The CPU was an AMD\textsuperscript{\textregistered} Ryzen\textsuperscript{TM} 9 3900X 12-Core Processor (3.8GHz, cache: L1 -- 768KiB, L2 -- 6 MiB, L3 -- 64 MiB) and 32GiB of RAM were available (2 x Crucial Ballistix Sport Red DDR4 16GB 2.4GHz).
The operating system used was Ubuntu 20.04 LTS (Linux 5.4.0-42-generic).
Hyper-Threading was disabled.
Each run executed on a single thread, and no runs executed simultaneously.
The computer did not run any other CPU bound task during the experiments.
The exact version of the code used is available online (\url{https://github.com/henriquebecker91/GuillotineModels.jl/tree/0.2.4}), and it was run using Julia 1.4.2~\cite{julia} with JuMP 0.20.1~\cite{JuMP} and Gurobi 9.0.2~\cite{gurobi}.
The following Gurobi parameters had non-default values: \verb+Threads+~\(= 1\); \verb+Seed+~\(= 1\); \verb+MIPGap+~\(= 10^{-6}\) (to guarantee optimality); and \verb+TimeLimit+~\(= 10800\) (i.e., three hours).
The next section explains the rationale for using \verb+Method+~\(= 2\) (i.e., barrier) to solve the root node relaxation of the final built model; and \verb+Method+~\(= 1\) (i.e., dual simplex) inside pricing (if pricing is enabled).

\subsection{The choice of LP algorithm}
\label{sec:lp_method}

Both \cite{furini:2016} and \cite{dimitri_thesis} do not specify the algorithm used for solving the MILP root node relaxation and, if pricing is enabled, for solving some LP models (upper bound computation) and the MILP root node relaxation of the \emph{restricted priced} model.
As we use Gurobi, we are discussing the \verb+Method+ parameter (for LP models and MILP root node relaxations), and not the \verb+NodeMethod+ parameter (for non-root nodes).
The choice of the algorithm can drastically impact running times.
A preliminary experiment included all LP algorithms available in Gurobi.
\autoref{tab:lp_method_comparison} presents the data of the two algorithms selected for use.
They are the \emph{Dual Simplex} and the \emph{Barrier}.

The runs use the \emph{faithful} implementation, with \emph{Cut-Position} and \emph{Redundant-Cut} enabled, in its \emph{priced} (Priced PP-G2KP in~\cite{furini:2016,dimitri_thesis}) and \emph{not priced} (PP-G2KP in~\cite{furini:2016,dimitri_thesis}) variants.
For convenience, we limited the experiment to a few instances.
This subset consists of all instances for which the \emph{Complete PP-G2KP Model} finds the optimal solution within the time limit in~\cite{furini:2016} (Table 2).
If pricing is disabled, the root node relaxation contributes for the majority of the running time.
This characteristic makes them a good choice for this experiment.

\begin{table}
\caption{Comparison of LP-solving algorithms used inside solving procedure.}
\begin{tabular}{@{\extracolsep{4pt}}lrrrrrrr@{}}
\hline\hline
Instance & \multicolumn{3}{c}{Dual Simplex} & \multicolumn{3}{c}{Barrier} & DS + B \\\cline{2-4}\cline{5-7}
& N. P. & R. \% & Priced & N. P. & R. \% & Priced & Priced \\\hline
CU1 & 27.37 & 92.11 & 3.79 & 24.18 & 94.68 & 3040.82 & \textbf{3.58} \\
STS4 & 93.49 & 89.88 & 48.80 & 49.94 & 77.32 & 7851.30 & \textbf{47.75} \\
STS4s & 103.20 & 94.92 & 39.29 & 43.74 & 86.34 & 8470.41 & \textbf{38.36} \\
gcut9 & 226.68 & 72.29 & \textbf{3.92} & 51.48 & 85.77 & 2060.04 & 4.01 \\
okp1 & 51.95 & 84.18 & 38.89 & \textbf{32.41} & 67.78 & -- & 38.79 \\
okp4 & 98.25 & 93.35 & 144.30 & \textbf{72.09} & 92.31 & -- & 141.53 \\
okp5 & 178.13 & 89.89 & 252.09 & \textbf{96.38} & 67.24 & -- & 239.44 \\\hline\hline
\end{tabular}
\label{tab:lp_method_comparison}
\end{table}

In \autoref{tab:lp_method_comparison}, \emph{Dual Simplex} and \emph{Barrier} indicate the respective algorithm was used for all LPs and root node relaxations;
and \emph{DS + B} means that \emph{Dual Simplex} was used to solve all LPs inside the pricing phase and \emph{Barrier} was used to solve the root node relaxation of the final model.
The columns \emph{N. P.} (\emph{Not Priced}) and \emph{Priced} display the time to solve (in seconds) using the aforementioned variant.
The columns \emph{R.\%} refer to the per cent of the time spent by \emph{Not Priced} in the root node relaxation of the final model.

The following conclusions can be derived from \autoref{tab:lp_method_comparison}.
Using the \emph{Barrier} algorithm in the pricing phase is not viable.
This impracticality happens because the pricing phase includes an iterative variable pricing phase.
This iterative phase repeatedly adds variables to one LP model and solve it again.
The \emph{Barrier} algorithm solves every LP from scratch;
the \emph{Dual Simplex} reuses the previous basis and saves considerable effort.
However, \emph{Barrier} performs better if there is no previous base to reuse.
Consequently, the configuration chosen was \emph{Dual Simplex} for the pricing phase, and \emph{Barrier} for the root relaxation of the final model.

\subsection{Comparison of \emph{faithful} against \emph{original}}
\label{sec:faithful_reimplementation}

Without a reimplementation of \emph{original}, any comparison would need to be made directly against the data in~\cite{dimitri_thesis}.
However, such comparison would hardly be fair, as it compares across machines, solvers, and programming languages.
Also, for example, it does not allow us to assess the benefits of applying the \emph{plate-size normalization} procedure to the \emph{original} formulation.
The purpose of this section is to show that \emph{faithful} may be fairly used in place of \emph{original}.
For this purpose, \autoref{tab:faithful_reimplementation} compares the number of model variables and number of plates of the diverse model variants presented in~\cite{furini:2016,dimitri_thesis} (using the same 59 instances).
The number of enumerated plates has a strong correlation to the number of constraints in the model.
Both~\cite{furini:2016} and~\cite{dimitri_thesis} present the number of plates and not the number of constraints.
To simplify the comparison, we do the same.

The \emph{Priced PP-G2KP} runs in~\cite{furini:2016,dimitri_thesis} had three time limits of one hour to solve: the restricted model (i.e., obtaining a lower bound); the iterative variable pricing (i.e., obtaining an upper bound); the final model.
Such configuration always generates a final model.
However, it also has two drawbacks:
(i) the computer performance may define the answer given in the first two phases, affecting the size of the final model (and making it harder to make a fair comparison);
(ii) if the restricted model, or the iterated variable pricing, cannot be done in one hour, then the final model will probably hit the time limit too -- in~\cite{furini:2016}, every run that hits one of the two first time limits also hits the third time limit.
We chose to use a single three-hour time limit.

\autoref{tab:faithful_reimplementation} references the names used in~\cite{furini:2016,dimitri_thesis}.
The \emph{Complete PP-G2KP} is the formulation with all optional procedures disabled, while the \emph{PP-G2KP} mean both \emph{Cut-Position} and \emph{Redundant-Cut} are enabled.
\emph{Restricted PP-G2KP} and its priced version are solved inside \emph{Priced PP-G2KP} runs.
The \emph{original} had no \emph{purge} phase after pricing.
Consequently, for the columns that refer to \emph{original}, the last row just repeats the data of the row above.

The sum of columns \emph{T. L.} (Time Limit) and \emph{E. R.} (Early Return) gives the number of instances excluded from consideration in the respective row.
Column \emph{T. L.} has the number of instances for which \emph{faithful} reached the time limit without generating the respective model variant -- these instances are: Hchl7s, okp2, and okp3.
The column \emph{E. R.} has the number of instances for which our reimplementation found an optimal solution before generating the respective model variant\footnote{
	If the lower and upper bounds found during pricing are the same, then the optimal solution was found before generating the final model.
	The instances in which this happened for an unrestricted solution are 3s, A1s, CU1, CU2, W, cgcut1, and wang20.
	The instance A1s presented this behaviour already in the pricing of the restricted model.
}.
Columns \emph{O. \#v} and \emph{O. \#v} refer to \emph{original}.
Column \emph{O. \#v} (\emph{O. \#p}) presents the sum of variables (plates) for the instances in which \emph{faithful} generated a model.
Columns \emph{F. \%v} and \emph{F. \%p} refer to \emph{faithful}.
Column \emph{R. \%v} (\emph{R. \%p}) has the sum of variables (plates) in the generated models, as a percentage of the quantity obtained by the original implementation.

\begin{table}
\caption{Comparison of \emph{faithful} against \emph{original}.}
\begin{tabular}{lccrrrr}
\hline\hline
Variant & T. L. & E. R. & O. \#v & F. \%v & O. \#p & F. \%p\\\hline
Complete PP-G2KP & 0 & 0 & 156,553,107 & 100.00 & 1,882,693 & 100.00\\
Complete +Cut-Position & 0 & 0 & 103,503,930 & 99.99 & 1,738,263 & 100.01\\
Complete +Redundant-Cut & 0 & 0 & 121,009,381 & 109.94 & 1,882,693 & 100.00\\
PP-G2KP (CP + RC) & 0 & 0 & 74,052,541 & 120.05 & 1,738,263 & 100.01\\
Restricted PP-G2KP & 0 & 0 & 5,335,976 & 99.28 & 306,673 & 99.99\\
Priced Restricted PP-G2KP & 0 & 1 & 3,904,683 & 102.20 & 305,690 & 99.99\\
(no purge) Priced PP-G2KP & 3 & 7 & 14,619,460 & 93.74 & 1,642,382 & 100.01\\
Priced PP-G2KP & 3 & 7 & 14,619,460 & 31.92 & 1,642,382 & 25.55\\\hline\hline
\end{tabular}
\label{tab:faithful_reimplementation}
\end{table}

The following conclusions can be derived from \autoref{tab:faithful_reimplementation}.
All variants, except \emph{Priced PP-G2KP}, are within \(\pm0.01\)\% of the expected number of plates (and, consequently, of constraints).
The \emph{Complete PP-G2KP}, \emph{Complete +Cut-Position}, and \emph{Restricted PP-G2KP} are within \(\pm1\)\% of the expected number of variables.
The number of variables in both \emph{Complete +Redundant-Cut} and \emph{PP-G2KP (CP + RC)} is \(10\sim20\)\% larger than expected.
Our reimplementation of \emph{Redundant-cut} reduction seems responsible for both deviations.
However, it follows closely the description given in~\cite{dimitri_thesis}.
The number of variables and plates in \emph{Priced} variants is not entirely deterministic.
The number of variables of \emph{Priced} variants is either slightly above (\(+2\)\%) or lower (\emph{\(-6\sim68\)\%}).

For all non-\emph{priced} variants, the fraction of the running time spent in the model generation is negligible.
Consequently, the comparison presented in~\autoref{tab:faithful_reimplementation} is sufficient.
We cannot say the same for the \emph{priced} variants.
\cite{furini:2016,dimitri_thesis} does not report the size of the multiple LP models solved inside the iterative pricing (a phase of the pricing).
For instances in which \emph{original} and \emph{faithful} executed all phases of pricing and solved the final model, the \emph{original} spent 34.35\% of its time in the iterative pricing phase, while \emph{faithful} spent 61.69\%.
It is hard to pinpoint the source of this discrepancy.
One possible explanation is that, in \emph{original}, other phases took more time than they took in \emph{faithful}.
For example, \emph{faithful} uses the \emph{barrier} algorithm for the root node relaxation of the final model, which reduces the percentage of time spent in this phase.
Nevertheless, for the subset of the instances aforementioned, the total time spent by \emph{faithful} was about 13\% of the time spent by \emph{original}.
While the difference between machines and solvers does not allow us to infer much from that figure, we believe that the magnitude of the difference guarantees that we are not making a gross misrepresentation.

\subsection{Comparison of \emph{faithful} against \emph{enhanced}}
\label{sec:comparison}

The primary purpose of this section is to evaluate our contributions to the state of the art.
Our contributions are the \emph{normalize} reduction (i.e., the plate-size normalization presented in~\autoref{sec:psn}) and the \emph{enhanced} formulation (presented in \autoref{sec:enhanced}).
The state of the art consists in a formulation (\emph{Complete PP-G2KP}), two reductions (\emph{Cut-Position} and \emph{Redundant-Cut}), and a pricing procedure presented in~\cite{furini:2016,dimitri_thesis}.
In this section, we use our reimplementation of \emph{Complete PP-G2KP} named \emph{faithful} (to distinguish from the data of the \emph{original}).
We also reimplemented the reductions and the pricing procedure, but as \emph{enhanced} may also enable them, we avoid labelling these procedures as \emph{faithful} as to avoid confusion.

The \emph{faithful} and \emph{enhanced} formulations cannot be combined.
However, both allow enabling any combination of the optional procedures.
The only exception is \emph{Redundant-Cut}, which is unnecessary for \emph{enhanced} and, therefore, never applied to it.
Outside of this exception, in this section, \emph{Redundant-Cut} and \emph{Cut-Position} are always enabled.
These reductions never increase the number of variables (or constraints), cost a negligible amount of computational effort, and were already discussed in~\cite{furini:2016,dimitri_thesis}.

We also examine the effects of our \emph{purge} procedure and warm-starting the non-\emph{priced} model.
The deterministic heuristic used to MIP-start the non-\emph{priced} models is the same used in the \emph{restricted priced} model solved inside the pricing procedure.

The meaning of the columns in~\autoref{tab:contribution} follow:
\emph{T. T.} (Total Time) -- sum of the time spent in all instances including timeouts, in seconds;
\emph{\#e} (early) -- number of instances in which pricing found an optimal solution (and, consequently, did not generate a final model);
\emph{\#m} (modeled) -- number of instances that generated a final model;
\emph{\#s} (solved) -- number of solved instances;
\emph{\#b} (best) -- number of instances that the respective variant solved faster than any other variant;
\emph{S. T. T.} (Solved Total Time) -- same as Total Time but excluding runs ended by time or memory limit;
\emph{\#variables} (\emph{\#plates}) -- sum of the variables (plates) in all generated final models (see column~\emph{\#m}).
The first row (Faithful) has two runs that ended in memory exhaustion.
We count the time of these runs as they were timeouts.

\begin{table}
\rowcolors{1}{white}{gray-table-row}
\caption{Comparison of \emph{faithful} vs. \emph{enhanced} over the 59 instances used in~\cite{dimitri_thesis}.}
\begin{tabular}{lrrrrrrrr}
\hline\hline
Variant & T. T. & \#e & \#m & \#s & \#b & S. T. T. & \#variables & \#plates \\\hline
Faithful & 106,057 & -- & 59 & 53 & 0 & 41,257 & 88,901,964 & 1,738,366 \\
Enhanced & 25,538 & -- & 59 & 58 & 2 & 14,738 & 3,216,774 & 231,836 \\
F. +Normalizing & 60,078 & -- & 59 & 56 & 0 & 27,678 & 60,316,964 & 610,402 \\
E. +Normalizing & 14,169 & -- & 59 & 59 & 52 & 14,169 & 2,733,125 & 145,157 \\
F. +N. +Warming & 60,542 & -- & 59 & 56 & 0 & 28,142 & 60,316,964 & 610,402 \\
E. +N. +Warming & 9,778 & -- & 59 & 59 & 4 & 9,778 & 2,733,125 & 145,157 \\
Priced F. +N. +W. & 49,919 & 8 & 50 & 55 & 0 & 6,719 & 3,210,857 & 174,214 \\
Priced E. +N. +W. & 9,108 & 8 & 51 & 59 & 1 & 9,108 & 600,778 & 64,904 \\
P. F. +N. +W. -Purge & 50,054 & 8 & 50 & 55 & 0 & 6,854 & 8,072,810 & 544,892 \\
P. E. +N. +W. -Purge & 9,209 & 8 & 51 & 59 & 0 & 9,209 & 1,021,526 & 134,102 \\\hline\hline
\end{tabular}
\label{tab:contribution}
\end{table}

Considering the data from~\autoref{tab:contribution} we can state that:
\begin{enumerate}
\item \emph{enhanced} solves more instances than \emph{faithful} (using at most 24\% of its time);
\item the number of variables of `Enhanced' is almost the same as `Priced F. +N. +W.';
\item between `Enhanced' and `Priced F. +N. +W.' the former has better results;
\item \emph{normalize} further reduces variables by \(14\sim32\)\% and plates by \(37\sim65\)\%;
\item MIP-starting \emph{enhanced} makes its slightly slower in 52 instances;
\item MIP-starting \emph{enhanced} saves more than one hour in the other 7 instances;
\item any benefit from MIP-start in `F. +N. +Warming' was negated by its timeouts;
\item \emph{purge} greatly reduces the model size but has almost no effect on running time;
\item the effects of applying \emph{pricing} to \emph{enhanced} are not much better than \emph{purge};
\item applying \emph{pricing} to \emph{faithful} is positive overall but loses one solved instance.
\end{enumerate}

In~\autoref{tab:time_fractions}, \emph{Time} is the sum of all time (in seconds) spent in the 47 instances that had all phases executed by all four variants considered.
These are the same 47 indicated in row \emph{Priced F. +N. +W.} of \autoref{tab:contribution}.
From the 59 instances dataset, 4 had timeout (Hchl4s, Hchl7s, okp2, and okp3), and 8 found an optimal solution inside pricing (3s, A1s, CU1, CU2, W, cgcut1, okp4, and wang20).
All remaining columns present percentages of the time spent in a specific phase:
\emph{E} -- enumeration of cuts and plates (and all reductions);
\emph{H} -- restricted heuristic used to warm-start the restricted priced model;
\emph{RP} -- restricted pricing (not including the heuristic time);
\emph{IP} -- iterative pricing;
\emph{FP} -- final pricing;
\emph{LP} -- root node relaxation of the final model;
\emph{BB} -- branch-and-bound over the final model.

\begin{table}
\rowcolors{1}{white}{gray-table-row}
\caption{Fraction of the total time spent in each step (only runs that executed all steps).}
\begin{tabular}{lrrrrrrrrr}
\hline\hline
Variant & Time & E~\% & H~\% & RP~\% & IP~\% & FP~\% & LP~\% & BB~\% \\\hline
Priced Faithful +N. +W. & 6,632 & 0.12 & 0.38 & 26.16 & 57.36 & 2.91 & 4.56 & 8.29 \\
Priced Enhanced +N. +W. & 1,178 & 0.03 & 2.18 & 50.89 & 23.66 & 0.46 & 2.70 & 19.95 \\
P. F. +N. +W. -Purge & 6,766 & 0.11 & 0.37 & 26.00 & 57.03 & 2.81 & 5.12 & 8.45 \\
P. E. +N. +W. -Purge & 1,185 & 0.03 & 2.18 & 50.70 & 23.64 & 0.46 & 2.83 & 20.09 \\\hline\hline
\end{tabular}
\label{tab:time_fractions}
\end{table}

Considering the data from~\autoref{tab:time_fractions} we can state that:
\begin{enumerate}
\item both \emph{BB} and \emph{LP} phases are slightly faster with \emph{purge} as expected;
\item both \emph{E} and \emph{H} phases are almost negligible (at most 2\% with \emph{H} in \emph{enhanced});
\item together the \emph{RP} and \emph{IP} phases account for \(74.5\sim83.5\)\%;
\item \emph{RP} and \emph{IP} swap percentages between \emph{enhanced} and \emph{faithful};
\item \emph{faithful} shows some overhead in all phases strongly affected by model size.
\end{enumerate}

\subsection{Evaluating \emph{enhanced} against harder instances}
\label{sec:new_results}

The purposes of the experiment described in this section are:
(i) to show the limitations of the \emph{enhanced} formulation against more challenging instances;
(ii) to provide better bounds and new proven optimal values for such instances.

\cite{velasco:2019} proposes a set of 80 hard instances to test the limitations of their bounding procedures; we use these instances in this section.
Only two variants were executed for this experiment, the \emph{priced} and non-\emph{priced} versions of \emph{enhanced} with \emph{Cut-Position}, \emph{normalize}, and \emph{MIP-start} enabled.
We also present the results for the \emph{restricted priced} variant because it executes inside \emph{priced} (the same reductions apply to it).
\autoref{tab:velasco_summary} presents a summary of all runs, and \autoref{tab:velasco_new_results} presents the improved bounds and solved instances.

For this experiment, Gurobi was allowed to use the 12 physical cores of our machine.
Gurobi distributes the effort of the B\&B phase equally among all cores.
However, solving an LP (as a root node relaxation, or not) calls barrier, primal simplex, and dual simplex.
Each of these three uses a single thread, and Gurobi stops when the first of them finish.

\autoref{tab:velasco_summary} columns are:
\emph{C.} -- instance class (described in~\cite{velasco:2019}, 20 instances each);
\emph{Variant} -- the solving method employed;
\emph{\#m} (modeled) -- number of instances in which the model was built before timeout;
\emph{Avg. \#v} and \emph{Avg. \#p} -- the average number of variables and plates in the \emph{\#m} instances that generated a final model for the respective variant;
\emph{T. T.} (Total Time) -- sum of the time spent in all instances in seconds, including timeouts;
\emph{\#s} (solved) -- number of instances solved;
\emph{Avg. S. T.} (Avg. Solved Time) -- as total time but excludes timeouts and divides by \emph{\#s}.
Averages were used instead of simple sums because the very different number of generated and solved models made the sums misleading.

\begin{table}
\caption{Summary table for the instances proposed in~\cite{velasco:2019}.}
\begin{tabular}{lrrrrrrr}
\hline\hline
C. & Variant & \#m & Avg. \#v & Avg. \#p & T. T. & \#s & Avg. S. T. \\\hline
\multirow{3}{*}{1} & Not Priced & 20 & 1,787,864.55 & 22,316.50 & 172,574 & 5 & 2,114.85 \\
                   & Restricted Priced & 13 & 467,692.15 & 17,139.00 & 180,051 & 5 & 3,610.29 \\
\vspace{1.5mm}     & Priced & 5 & 264,315.80 & 11,978.40 & 196,733 & 3 & 4,377.77 \\
\multirow{3}{*}{2} & Not Priced & 20 & 1,533,490.70 & 18,638.50 & 167,973 & 5 & 1,194.68 \\
                   & Restricted Priced & 20 & 453,159.70 & 18,638.30 & 155,184 & 8 & 3,198.11 \\
\vspace{1.5mm}     & Priced & 8 & 394,613.88 & 9,735.50 & 178,812 & 4 & 1,503.01 \\
\multirow{3}{*}{3} & Not Priced & 20 & 2,895,300.75 & 33,249.40 & 171,155 & 5 & 1,831.11 \\
                   & Restricted Priced & 10 & 431,913.00 & 15,895.80 & 174,569 & 5 & 2,513.80 \\
\vspace{1.5mm}     & Priced & 5 & 372,597.00 & 13,287.80 & 179,712 & 4 & 1,728.08 \\
\multirow{3}{*}{4} & Not Priced & 20 & 3,201,374.45 & 35,197.10 & 167,776 & 7 & 3,910.89 \\
                   & Restricted Priced & 10 & 497,802.20 & 17,011.00 & 197,047 & 2 & 1,323.65 \\
                   & Priced & 2 & 211,093.00 & 14,227.00 & 199,477 & 2 & 2,538.79 \\\hline\hline
\end{tabular}
\label{tab:velasco_summary}
\end{table}

Concerning the data from~\autoref{tab:velasco_summary}, we want to highlight some unexpected results:
(i) the total number of instances solved by the \emph{restricted priced} was slightly smaller than non-\emph{priced}, even with non-\emph{priced} solving the harder \emph{unrestricted} problem;
(ii) many runs reached time limit without solving the continuous relaxation of the \emph{restricted} model (necessary for creating \emph{restricted priced} model);
(iii) non-\emph{priced} solved more instances than \emph{priced} in all cases.
Ideally, the pricing procedure would significantly reduce the size of the model and, consequently, the root node relaxation and B\&B phases would take much less time to solve.
However, the gain in decreasing the size of the (already reduced) \emph{enhanced} model further does not seem to compensate for the cost of solving hard LPs more than once.
Also, previous sections have shown that reducing the model size does not guarantee that the running time will be reduced by the same magnitude.

The purpose of \autoref{tab:velasco_new_results} is to allow querying the exact values for specific instances.
Even so, there are some gaps to fill.
For the instances presented in \autoref{tab:velasco_new_results},
the min / mean / max gap between the heuristic lower bound and the final lower bound were: 0.38 / 18.08 / 37.03 (non-\emph{priced}); 0.68 / 20.62 / 37.29 (\emph{restricted priced}); 9.17 / 19.38 / 32.24 (\emph{priced}).
In other words, no solution, or best bound, was given by the heuristic, and most of the time, its solution was considerably improved.
For the reader convenience, we can also summarize that our experiment has:
proved 22 unrestricted optimal values (5 already proven by~\cite{velasco:2019}, confirming their results);
proved 22 restricted optimal values (in an overlapping but distinct subset of the instances);
improved lower bounds for 25 instances;
improved upper bounds for 58 instances.

\autoref{tab:velasco_new_results} groups lower and upper bounds that are valid for the unrestricted problem.
Column \emph{RP UB} (restricted priced upper bound) is kept separate as it is not a valid bound for the unrestricted problem.
Bold indicates the best unrestricted bounds for the instance.
For the same instance and variant, if the LB and the UB are the same, both values are underlined.
The sub-headers mean:
\emph{RP} -- Restricted Priced (solved inside \emph{P} runs);
\emph{P} -- Priced;
\emph{NP} -- Not Priced;
\emph{V\&U} -- obtained by Velasco and Uchoa in~\cite{velasco:2019}.

\begin{table}
\let\mc\multicolumn
\rowcolors{3}{white}{gray-table-row}
\caption{Instances solved (restricted or unrestricted) or with improved bounds.}
\begin{tabular}{lrrrrrrrr}
\hline\hline
\hiderowcolors
Instance & \mc4c{Lower Bounds for Unrestricted} & RP UB & \mc3c{Upper Bounds for Unr.} \\\cline{2-5}\cline{7-9}
 & \mc1c{RP} & \mc1c{P} & \mc1c{NP} & \mc1c{V\&U} & & \mc1c{P} & \mc1c{NP} & \mc1c{V\&U} \\\hline
\showrowcolors
P1\_100\_200\_25\_1 & \underline{\textbf{27,251}} & \underline{\textbf{27,251}} & \underline{\textbf{27,251}} & \textbf{27,251} & \underline{27,251} & \underline{\textbf{27,251}} & \underline{\textbf{27,251}} & 27,340 \\
P1\_100\_200\_25\_2 & \underline{\textbf{25,090}} & \textbf{25,090} & \textbf{25,090} & 24,870 & \underline{25,090} & 25,403 & \textbf{25,389} & 25,522 \\
P1\_100\_200\_25\_3 & \underline{\textbf{25,730}} & \textbf{25,730} & \textbf{25,730} & \textbf{25,730} & \underline{25,730} & 25,974 & \textbf{25,909} & 26,088 \\
P1\_100\_200\_25\_4 & \underline{26,732} & \underline{\textbf{26,896}} & \underline{\textbf{26,896}} & 26,769 & \underline{26,732} & \underline{\textbf{26,896}} & \underline{\textbf{26,896}} & 27,051 \\
P1\_100\_200\_25\_5 & \textbf{26,152} & -- & \textbf{26,152} & 25,772 & 26,565 & -- & \textbf{26,617} & 26,857 \\
P1\_100\_200\_50\_1 & 28,388 & -- & \underline{\textbf{28,440}} & 28,388 & 28,504 & -- & \underline{\textbf{28,440}} & 28,558 \\
P1\_100\_200\_50\_2 & \underline{\textbf{26,276}} & \underline{\textbf{26,276}} & \underline{\textbf{26,276}} & \textbf{26,276} & \underline{26,276} & \underline{\textbf{26,276}} & \underline{\textbf{26,276}} & 26,326 \\
P1\_100\_200\_50\_3 & \textbf{27,192} & -- & \textbf{27,192} & 27,165 & 27,536 & -- & \textbf{27,483} & 27,679 \\
P1\_100\_200\_50\_4 & 28,058 & -- & \textbf{28,095} & 27,977 & 28,345 & -- & \textbf{28,340} & 28,388 \\
P1\_100\_200\_50\_5 & \textbf{27,722} & -- & \underline{\textbf{27,722}} & 27,603 & 27,930 & -- & \underline{\textbf{27,722}} & 28,009 \\
P1\_100\_400\_25\_1 & 53,247 & -- & 53,008 & \textbf{53,904} & 54,540 & -- & \textbf{54,707} & 55,038 \\
P1\_100\_400\_25\_2 & -- & -- & 41,275 & \textbf{44,581} & -- & -- & \textbf{47,091} & 47,097 \\
P1\_100\_400\_25\_3 & 42,748 & -- & 46,222 & \textbf{47,455} & \textbf{\large \textasteriskcentered} & -- & \textbf{49,371} & 49,473 \\
P1\_100\_400\_25\_4 & -- & -- & 38,567 & \textbf{40,517} & -- & -- & \textbf{46,069} & 46,078 \\
P1\_100\_400\_25\_5 & 44,482 & -- & \textbf{53,220} & 53,205 & \textbf{\large \textasteriskcentered} & -- & 54,120 & \textbf{54,063} \\
P1\_100\_400\_50\_1 & -- & -- & 53,831 & \textbf{55,856} & -- & -- & \textbf{56,897} & 57,074 \\
P1\_100\_400\_50\_2 & -- & -- & 40,440 & \textbf{48,373} & -- & -- & \textbf{51,754} & 51,893 \\
P1\_100\_400\_50\_4 & -- & -- & \textbf{55,107} & 52,708 & -- & -- & \textbf{55,654} & 55,661 \\
P1\_100\_400\_50\_5 & -- & -- & \textbf{53,749} & 53,502 & -- & -- & \textbf{55,005} & 55,454 \\
P2\_200\_100\_25\_1 & \underline{\textbf{21,494}} & \underline{\textbf{21,494}} & \underline{\textbf{21,494}} & \underline{\textbf{21,494}} & \underline{21,494} & \underline{\textbf{21,494}} & \underline{\textbf{21,494}} & \underline{\textbf{21,494}} \\
P2\_200\_100\_25\_2 & \underline{25,244} & \underline{\textbf{25,413}} & \underline{\textbf{25,413}} & \textbf{25,413} & \underline{25,244} & \underline{\textbf{25,413}} & \underline{\textbf{25,413}} & 25,648 \\
P2\_200\_100\_25\_3 & \underline{25,282} & \textbf{25,397} & \textbf{25,397} & \textbf{25,397} & \underline{25,282} & \textbf{25,640} & 25,647 & 25,723 \\
P2\_200\_100\_25\_4 & 25,729 & -- & \textbf{25,734} & 25,437 & 26,181 & -- & \textbf{26,239} & 26,898 \\
P2\_200\_100\_25\_5 & \underline{26,211} & \textbf{26,413} & \underline{\textbf{26,413}} & 26,220 & \underline{26,211} & 26,728 & \underline{\textbf{26,413}} & 26,898 \\
P2\_200\_100\_50\_1 & \textbf{25,679} & -- & 25,626 & 25,627 & 26,233 & -- & \textbf{26,282} & 26,447 \\
P2\_200\_100\_50\_2 & \underline{\textbf{27,801}} & \underline{\textbf{27,801}} & \underline{\textbf{27,801}} & 27,789 & \underline{27,801} & \underline{\textbf{27,801}} & \underline{\textbf{27,801}} & 27,943 \\
P2\_200\_100\_50\_3 & \underline{27,435} & \textbf{27,453} & \textbf{27,453} & \textbf{27,453} & \underline{27,435} & 27,584 & \textbf{27,579} & 27,596 \\
P2\_200\_100\_50\_4 & 27,395 & -- & \textbf{27,439} & 27,362 & 27,668 & -- & \textbf{27,704} & 27,718 \\
P2\_200\_100\_50\_5 & \underline{\textbf{29,386}} & \underline{\textbf{29,386}} & \underline{\textbf{29,386}} & \underline{\textbf{29,386}} & \underline{29,386} & \underline{\textbf{29,386}} & \underline{\textbf{29,386}} & \underline{\textbf{29,386}} \\
P2\_400\_100\_25\_1 & 49,327 & -- & \textbf{49,947} & 49,026 & 50,218 & -- & \textbf{50,365} & 51,006 \\
P2\_400\_100\_25\_2 & 48,312 & -- & \textbf{48,542} & 47,773 & 49,268 & -- & \textbf{49,315} & 49,908 \\
P2\_400\_100\_25\_3 & \textbf{46,970} & -- & 46,860 & 45,406 & 47,113 & -- & \textbf{47,204} & 48,938 \\
P2\_400\_100\_25\_4 & \textbf{51,051} & -- & 49,847 & 49,521 & 51,526 & -- & \textbf{51,600} & 52,229 \\
P2\_400\_100\_25\_5 & \textbf{49,620} & -- & 48,832 & 47,403 & 50,440 & -- & \textbf{50,580} & 54,248 \\
P2\_400\_100\_50\_1 & \underline{54,550} & 54,550 & \textbf{54,679} & 52,890 & \underline{54,550} & 54,981 & \textbf{54,916} & 55,629 \\
P2\_400\_100\_50\_2 & \textbf{54,821} & -- & 54,768 & 53,492 & 55,183 & -- & \textbf{55,181} & 55,543 \\
P2\_400\_100\_50\_3 & 54,141 & -- & \textbf{54,747} & 54,216 & 55,537 & -- & \textbf{55,709} & 56,065 \\
P2\_400\_100\_50\_4 & 53,375 & -- & \textbf{54,240} & 48,649 & 54,857 & -- & \textbf{54,987} & 55,604 \\
P2\_400\_100\_50\_5 & \textbf{53,763} & -- & 53,541 & 50,047 & 54,893 & -- & \textbf{54,918} & 55,471 \\
P3\_150\_150\_25\_1 & \underline{29,896} & \underline{\textbf{29,989}} & \underline{\textbf{29,989}} & 29,896 & \underline{29,896} & \underline{\textbf{29,989}} & \underline{\textbf{29,989}} & 30,005 \\
P3\_150\_150\_25\_2 & \textbf{29,345} & -- & 29,196 & 29,101 & 29,906 & -- & 29,965 & \textbf{29,961} \\
P3\_150\_150\_25\_3 & \underline{\textbf{30,286}} & \underline{\textbf{30,286}} & \underline{\textbf{30,286}} & \textbf{30,286} & \underline{30,286} & \underline{\textbf{30,286}} & \underline{\textbf{30,286}} & 30,327 \\
P3\_150\_150\_25\_5 & \underline{\textbf{31,332}} & \textbf{31,332} & \textbf{31,332} & 30,924 & \underline{31,332} & 31,715 & \textbf{31,682} & 31,839 \\
P3\_150\_150\_50\_1 & \underline{31,377} & \underline{\textbf{31,701}} & \underline{\textbf{31,701}} & \textbf{31,701} & \underline{31,377} & \underline{\textbf{31,701}} & \underline{\textbf{31,701}} & 31,892 \\
P3\_150\_150\_50\_2 & 30,846 & -- & \textbf{30,884} & \textbf{30,884} & 31,110 & -- & \textbf{31,008} & 31,115 \\
P3\_150\_150\_50\_3 & \underline{32,037} & \underline{\textbf{32,121}} & \underline{\textbf{32,121}} & 32,050 & \underline{32,037} & \underline{\textbf{32,121}} & \underline{\textbf{32,121}} & 32,240 \\
P3\_150\_150\_50\_4 & \textbf{31,925} & -- & \underline{\textbf{31,925}} & \textbf{31,925} & 32,210 & -- & \underline{\textbf{31,925}} & 32,070 \\
P3\_150\_150\_50\_5 & \textbf{31,631} & -- & 31,521 & 31,448 & 31,857 & -- & \textbf{31,896} & 31,901 \\
P3\_250\_250\_25\_1 & -- & -- & 51,027 & \textbf{58,480} & -- & -- & \textbf{60,548} & 60,611 \\
P3\_250\_250\_25\_2 & -- & -- & 63,646 & \textbf{68,070} & -- & -- & \textbf{73,316} & 73,339 \\
P3\_250\_250\_50\_1 & -- & -- & 59,072 & \textbf{67,603} & -- & -- & \textbf{76,117} & 76,341 \\
P3\_250\_250\_50\_2 & -- & -- & 62,772 & \textbf{75,569} & -- & -- & \textbf{82,644} & 82,666 \\
P4\_150\_150\_25\_1 & 30,870 & -- & \underline{\textbf{30,923}} & \textbf{30,923} & 31,094 & -- & \underline{\textbf{30,923}} & 31,130 \\
P4\_150\_150\_25\_2 & 30,576 & -- & \underline{\textbf{30,687}} & 30,460 & 30,786 & -- & \underline{\textbf{30,687}} & 30,931 \\
P4\_150\_150\_25\_3 & 30,257 & -- & \underline{\textbf{30,352}} & \underline{\textbf{30,352}} & 30,501 & -- & \underline{\textbf{30,352}} & \underline{\textbf{30,352}} \\
P4\_150\_150\_25\_4 & \underline{30,055} & \underline{\textbf{30,106}} & \underline{\textbf{30,106}} & \underline{\textbf{30,106}} & \underline{30,055} & \underline{\textbf{30,106}} & \underline{\textbf{30,106}} & \underline{\textbf{30,106}} \\
P4\_150\_150\_25\_5 & \textbf{30,582} & -- & 30,102 & \textbf{30,582} & 30,952 & -- & \textbf{31,228} & 31,286 \\
P4\_150\_150\_50\_1 & \underline{\textbf{31,673}} & \underline{\textbf{31,673}} & \underline{\textbf{31,673}} & \underline{\textbf{31,673}} & \underline{31,673} & \underline{\textbf{31,673}} & \underline{\textbf{31,673}} & \underline{\textbf{31,673}} \\
P4\_150\_150\_50\_2 & 32,302 & -- & \underline{\textbf{32,317}} & \textbf{32,317} & 32,434 & -- & \underline{\textbf{32,317}} & 32,423 \\
P4\_150\_150\_50\_3 & 30,906 & -- & \textbf{30,913} & 30,882 & 31,500 & -- & \textbf{31,519} & 31,756 \\
P4\_150\_150\_50\_4 & 31,912 & -- & \underline{\textbf{31,961}} & 31,912 & 32,206 & -- & \underline{\textbf{31,961}} & 32,140 \\
P4\_150\_150\_50\_5 & \textbf{32,027} & -- & 31,845 & 31,864 & 32,331 & -- & \textbf{32,308} & 32,484 \\
P4\_250\_250\_25\_4 & -- & -- & 69,530 & \textbf{79,476} & -- & -- & \textbf{81,634} & 81,839 \\
P4\_250\_250\_50\_2 & -- & -- & 67,675 & \textbf{77,206} & -- & -- & \textbf{87,314} & 87,331 \\
P4\_250\_250\_50\_4 & -- & -- & 69,063 & \textbf{78,359} & -- & -- & \textbf{86,941} & 87,069 \\\hline\hline
\end{tabular}
\textbf{\large \textasteriskcentered} These runs hit the time limit at the very start of the upper bound computation and, consequently, they produced only large and irrelevant upper bounds, which we omit to keep the table formatting.
\label{tab:velasco_new_results}
\end{table}

\section{Conclusions}
\label{sec:conclusions}

The present work advances the state of the art on MILP formulations for the G2KP.
We improve the performance of one of the most competitive MILP formulations for the G2KP by at least one order of magnitude.
In the instance set selected by the original formulation, our enhanced formulation dominates the original formulation.
Concerning other competitive MILP formulations in the literature, we keep the advantage of tighter bounds the original formulation had over them, and greatly reduce the model size and running times for instances that these other formulations had the advantage.

In the experiments, we have already discussed some elementary inferences, for example: the limitations (and partial success) of our improved formulation against the most recent and challenging instances in the literature; and the impact on the performance caused by the LP-solving algorithm, by the specific changes we made, by MIP-starting the models, and by some procedures proposed together with the original model (i.e., pricing and some preprocessing reductions).
Here we present more general conclusions from a broader perspective.

\emph{We believe symmetry-breaking plays a significant part in the success of our enhanced formulation.}
In our experiments, we focus on the significant reduction of the model size because it is easier to measure.
However, in \autoref{sec:comparison}, by comparing formulations with and without the \emph{purge} procedure, we see that a significant reduction of the model size does not always lead to a significant reduction of running times.
In the case of the variables removed by the \emph{purge} procedure (which could never assume a nonzero value), it seems clear the solver was able to disregard them without the need of our explicit removal.
The same does not apply to the variables removed by our enhanced model, which could assume nonzero values and compose symmetric solutions.
A single extraction variable may replace many distinct sequences of cuts that would extract the same piece from the same slightly-larger plate.
We also believe our results suggest that clever dominance rules may considerably improve pseudo-polynomial models (which often have tight bounds but large formulations) before resorting to more complicated techniques (as the pricing procedure proposed in~\cite{furini:2016} or column generation techniques)

\emph{Limited parallelisation of solving LP models is becoming a bottleneck.}
Obtaining tighter bounds, even at the cost of larger model size, is often valuable.
Some recent examples of this trade-off are pseudo-polynomial models like ours, but exponential-sized models solved by column generation are a pervasive and older example of the same trade-off.
In our experiment focusing on finding new optimal solutions for hard instances, it became clear that this approach shifts computational effort from the massively parallelisable B\&B phase to the almost serial root node relaxation phase.
This effect postpones finding the first primal solution and diminishes the value in massive computer clusters.

Our suggestions for future works follow: adapt the formulation for closely related problem variants and compare to their state-of-the-art solving procedure; expand on the symmetry-breaking; search for more parallelisable ways of solving LPs; consider other frameworks besides the pricing framework of~\cite{furini:2016}.

\begin{acknowledgements}
We would like to thank Dimitri Thomopulos for informing us that the code used in~\cite{furini:2016,dimitri_thesis} was, unfortunately, not available anymore.
Also, we would like to thank Mateus Pereira Martin for sharing his partial reimplementation of the code above mentioned, while not used in this paper his code allowed us to double-check our own reimplementation.
\end{acknowledgements}

%
\section*{Declarations}
\subsection*{Funding}
This study was financed in part by the Coordenação de Aperfeiçoamento de Pessoal de Nível Superior - Brasil (CAPES) - Finance Code 001
\subsection*{Conflict of interest}
The authors declare that they have no conflict of interest.
\subsection*{Availability of data and material}
The Jupyter notebooks and raw CSVs used for generating the tables in this work are available in~\url{https://github.com/henriquebecker91/phd/tree/MPC-1/latex/revised_PPG2KP/notebooks}.
\subsection*{Code availability}
The code, in the specific version used, is available in~\url{https://github.com/henriquebecker91/GuillotineModels.jl/tree/0.2.4}. The code is in public domain (Unlicense).
\subsection*{Authors' contributions}
The authors' contributions may be summarised as:
\begin{enumerate}
\item an enhanced MILP formulation based on a previous state-of-the-art formulation;
\item the proof of correctness of the aforementioned enhancement;
\item empirical evidence of the better performance obtained by the enhancement;
\item the adaptation of a known reduction for our specific case;
\item empirical evidence of the gains obtained by this reduction in our case;
\item 17 new optimal solution values of hard instances (unrestricted problem);
\item 22 new optimal solution values of hard instances (restricted problem);
\item better lower bounds for 25 instances (unrestricted problem);
\item better upper bounds for 58 instances (unrestricted problem).
\end{enumerate}

\bibliographystyle{spmpsci}   
\bibliography{mybib} 

\end{document}